\documentclass[12pt]{article}  
\usepackage{amsmath,amssymb,color,amscd}

\def\seq#1#2#3{#1_{#2},\,\ldots,#1_{#3}}

\def\m{\mathfrak m}

\def\vv{{\underline{v}}}

\def\nuv{{\underline{\nu}}}
\def\tt{{\underline{t}}}

\def\ww{\underline{w}}

\newcommand{\Var}{\ensuremath{\mathcal{V}_{\mathbb{C}}}}
\def\1{\underline{1}}

\def\P{\mathbb P}

\def\L{\mathbb L}
\def\LLL{\mathbb L}
\def\Z{\mathbb Z}

\def\C{\mathbb C}
\def\S{\mathbb S}

\def\OO{{\cal O}}
\def\M{\widehat{\cal M}}

\def\S{\mathbb S}
\def\nunu{\underline{\nu}}

\newtheorem{theorem}{Theorem}

\newtheorem{corollary}{Corollary}

\newtheorem{proposition}{Proposition}

\newenvironment{definition}
{\smallskip\noindent{\bf Definition\/}:}{\smallskip\par}

\newenvironment{remark}
{\smallskip\noindent{\bf Remark\/}.}{\smallskip\par}
\newenvironment{remarks}
{\smallskip\noindent{\bf Remarks\/}.}{\smallskip\par}
\newenvironment{proof}
{\noindent{\bf Proof\/}.}{{ $\Box$}\smallskip\par}
\newenvironment{Proof}
{\noindent{\bf Proof\/}}{{ $\Box$}\smallskip\par}

\title{Hilbert function, generalized Poincar\'e series and topology of plane valuations}

\author{
A.~Campillo
\and F.~Delgado \and S.M.~Gusein-Zade
\thanks{
Math. Subject Class. 14B05, 16W70, 13A18.
Partially supported by the grant MTM2007-64704 (with the help of
FEDER Program) and MTM2012-36917-C03-01 / 02.
Third author is also partially supported by the
Russian government grant 11.G34.31.0005, RFBR--13-01-00755,
NSh--4850.2012.1 and Simons-IUM fellowship.
} }

\date{}
\begin{document}
\def\eps{\varepsilon}

\maketitle

\begin{abstract}
To a multi-index filtration (say, on the ring of germs of functions on a germ of a
complex analytic variety)
one associates several invariants: the Hilbert function, the Poincar\'e series, the
generalized Poincar\'e series,
and the generalized semigroup Poincar\'e series. The Hilbert function and the
generalized
Poincar\'e series are equivalent in the sense that each of them determines the
other one. We show that for
a filtration on the ring of germs of holomorphic functions in two variables defined
by a collection of plane
valuations both of them are equivalent to the generalized semigroup Poincar\'e
series and determine the topology
of the collection of valuations, i.e. the topology of its minimal resolution.
\end{abstract}

\section*{Introduction}\label{sec0}

Let $\OO_{X,0}$ be the ring of germs of functions on a germ $(X,0)$ of a complex
analytic variety. A (one-index) filtration
by vector subspaces on the ring $\OO_{X,0}$
$$
\OO_{X,0} = J(0)\supset J(1) \supset J(2)\supset \cdots
$$
can be described by the function $\nu: \OO_{X,0}\to \Z_{\ge 0}\cup\{\infty\}$ defined
by $\nu(g)=\sup \{i : g\in J(i)\}$. This function possesses the properties:
\begin{enumerate}
\item[1)] $\nu(\lambda g) = \nu(g)$ for $\lambda\in \C^*$,
$\nu(0)=\infty$;
\item[2)] $\nu(g_1+g_2)\ge \min \{\nu(g_1),\nu(g_2)\}$.
\end{enumerate}
Functions $\nu : \OO_{X,0}\to \Z_{\ge 0}\cup \{\infty\}$ with the properties 1) and
2) are called {\em order functions}. If, moreover,
\begin{enumerate}
\item[3)] $\nu(g_1 g_2)= \nu(g_1)+\nu(g_2)$,
\end{enumerate}
the function $\nu$ is a valuation on $\OO_{X,0}$.

A multi-index filtration on the ring $\OO_{X,0}$ will be defined by a
collection $\seq{\nu}1r$ of order functions on $\OO_{X,0}$:
for $\vv=(\seq v1r)\in \Z^r$
$$
J(\vv) = \{g\in \OO_{X,0} : \nu_i(g)\ge v_i \mbox{ for }
i=1,\ldots,r \}\;
.
$$
(It is sufficient to define $J(\vv)$ for $\vv\in \Z_{\ge 0}^{r}$.
However, below it will be convenient to assume it to be defined for all
$\vv\in\Z^r$).

\begin{remarks}
{\bf 1.} One can consider a different notion of a multi-index filtration defined by a
system of subspaces $J(\vv)$ numbered by $\vv\in \Z_{\ge 0}^r$ such that for
$\vv\ge \ww$, $J(\vv)\subset J(\ww)$
(if $\vv=(\seq{v}1r)$,  $\ww = (\seq{w}1r)$,
$\vv\ge \ww$ if and only if
$v_i\ge w_i$ for all $i=1,\ldots, r$). A multi-index filtration
in this sense is
defined by a collection of order functions if and only if for any $\vv$ and
$\ww$ in
$\Z_{\ge 0}^r$ one has that
$J(\vv) \cap J(\ww) = J(\max (\vv,\ww))$, where
$\max(\vv,\ww) = (\max (v_1,w_1), \ldots, \max(v_r,w_r))$
($\vv=(\seq{v}1r)$, $\ww = (\seq{w}1r)$).

{\bf 2.}
Below we shall also consider more general order functions
$\nu: \OO_{X,0}\to S\cup \{\infty\}$ with values in an ordered semigroup $S$ and
therefore filtrations indexed by
$\vv = (\seq v1r)\in \S = S_1\times \cdots\times S_r$.
(The semigroup $\S$ is partially ordered by the relation $\vv\ge
\ww$ if and only if
$v_i\ge w_i$ for all $i$).
\end{remarks}

To a multi-index filtration $\{J(\vv)\}$ ($\vv\in \Z^r$) one
associates several invariants:

{\bf 1.}
The {\em Hilbert function}: $h(\vv)=\dim \OO_{X,0}/J(\vv)$.
One can describe the Hilbert function by the power series
$\widetilde {H}(\tt) = \sum\limits_{\vv\in \Z_{\ge 0}^r}
h(\vv)\;\tt^{\,\vv}$
(here $\tt = (\seq t1r)$ and $\tt^{\,\vv} = t_1^{v_1}\cdots
t_r^{v_r}$) or by a sort of a Laurent series
$H(\tt) = \sum\limits_{\vv\in \Z^r} h(\vv)\;\tt^{\,\vv}$.
They are defined if the subspaces $J(\vv)$ have finite codimensions.

{\bf 2.}
The {\em Poincar\'e series} $P(\tt)$ defined in \cite{CDK} (see also \cite{IJM2003}):
\begin{equation}\label{eq0}
P(\tt) = \frac{L(\tt) \prod_{i=1}^r(t_i-1)}{t_1\cdots t_r -1}
\end{equation}
where $L(\tt) = \sum\limits_{\vv\in \Z^r} \dim
(J(\vv)/J(\vv+\1))\, \tt^{\,\vv}$, $\1=(1, \ldots, 1)$. (One has $\tt\cdot L(\tt)= (1-\tt)
H(t)$.)

The Poincar\'e series can be also defined as an integral with respect to the Euler
characteristic (see \cite{IJM2003}):
\begin{equation}\label{eq1}
P(\tt) = \int_{\P \OO_{X,0}} \tt^{\,\nunu(g)}\, d\chi\,,
\end{equation}
where $\nunu(g)=(\nu_1(g), \ldots, \nu_r(g))$, $t^{\infty}$ is assumed to be equal to zero.

If $(X,0)=(\C^2,0)$ and $\seq{\nu}1r$ are curve valuations of
rank one (type I.1 in the notations of \cite{IJM2010}) corresponding
to
irreducible plane curve singularities $(C_i,0)\subset (\C^2,0)$,
$i=1,\ldots,r$, then the Poincar\'e series $P(\tt)$ coincides
with the Alexander polynomial $\Delta^C(\tt)$ of the link
$L=C\cap S^3_{\varepsilon}$ of the curve $C=\bigcup_{i=1}^r C_i$ (see \cite{IJM2003}).

{\bf 3.}
The {\em generalized Poincar\'e series} $P_g(\tt; q)$ defined
in \cite{Monats} as the
same  integral as in equation (\ref{eq1}) with respect to the
generalized
Euler characteristic $\chi_g$ with values in the Grothendieck
ring $K_0(\Var)$
of quasi-projective varieties localized at $\L= [\C]$:
\begin{equation}\label{gen}
P_g(\tt; q) = \int_{\P \OO_{X,0}} \tt^{\,\nunu(g)} d\chi_g\,,
\end{equation}
where $q = \L^{-1}$.
(One has $P_g(\tt; q)\in \Z[q][[\seq t1r]]$: see \cite{Monats}.)

The generalized Poincar\'e series is related with
zeta-functions of curves defined over finite fields (see
\cite{MZ}).
It is conjectured that the generalized Poincar\'e series corresponding to a collection of plane curve valuations is
closely related with the generating series of the Heegard-Floer
homologies of the link $L = C\cap S^3_{\varepsilon}$ (see \cite{GN}).

In \cite{Monats} it was explained that the Hilbert function of a filtration defines
the Poincar\'e series
and the generalized Poincar\'e series of it. On the other hand the Poincar\'e
series does not define,
in general, the Hilbert function and the generalized Poincar\'e series: see Example
 in \cite{Monats} taken from \cite{CDK}.

Here we discuss relations between the Hilbert function and  the
generalized Poincar\'e series
in the general setting, i.e. for order functions with
values in ordered semigroups.
We  show that, if the Hilbert function is defined (i.e. if all the subspaces
$J(\vv)$ have finite codimensions),
these two invariants are equivalent in the sense that each
of them determines the other one.
Thus they keep
the same information about the filtration. However this
information
is encoded in different ways. In particular, as it was mentioned,
for plane curve valuations the generalized Poincar\'e series
is in the form related to classical knot invariants of the
corresponding link.

\medskip
Now let $(X,0)$ be the plane $(\C^2,0)$ and let the order functions $\nu_i$ be
valuations on the ring $\OO_{\C^2,0}$
of germs of functions in two variables. The classification of the valuations on
$\OO_{\C^2,0}$ can be found in
\cite{Spiv} (see also \cite{IJM2010}).

It is well-known that the Alexander polynomial of a plane curve singularity,
coinciding with the Poincar\'e series of
the corresponding collection of curve valuations of rank one, determines the
topology of the curve singularity:
\cite{yamamoto} (see also \cite{FAOM}). In \cite{FAOM} it was
shown that the Poincar\'e series determines the topology
of a collection of divisorial valuations in the plane, i.e. the topology of its
minimal resolution. (In \cite{IJM2010}
this was generalized to an arbitrary collection of valuations on the ring
$\OO_{\C^2,0}$ does not containing
curve valuations of rank one.) On the other hand, an example in \cite{FAOM} shows
that the Poincar\'e series
does not determine the topology of an arbitrary collection of valuations on
$\OO_{\C^2,0}$.

Here we show that the generalized Poincar\'e series of a collection of plane
valuations determines the topology
of them. In particular this implies that the Hilbert function of a collection of
plane valuations determines its topology.
One can say that this is the main result of the paper formulated
in ``traditional" terms.

\section{Hilbert function and generalized Poincar\'e series}\label{sec1}

Let $S$ be a totally ordered abelian semigroup with the minimal element equal to zero. An
order function on $\OO_{X,0}$ with values in $S$ is  a map
$\nu: \OO_{X,0}\to S\cup \{\infty\}$ such that
\begin{enumerate}
\item[1)] $\nu(\lambda g) = \nu (g)$ for $\lambda\in \C^*$, $\nu(0)=\infty$;
\item[2)] $\nu(g_1 + g_2)\ge \min \{\nu(g_1), \nu(g_2)\}$.
\end{enumerate}

\begin{remarks}
{\bf 1)} If $\nu(g_1 g_2)= \nu(g_1)+\nu (g_2)$, the order function $\nu$ is a
valuation. In this section the property to be a valuation will not be essential.

{\bf 2)} All the content of the section is valid for order functions (and thus for
filtrations) on arbitrary complex vector spaces, e.g. on a module over $\OO_{X,0}$.
We describe the situation in $\OO_{X,0}$, not trying to formulate in the most general
context, for convenience.
\end{remarks}

A (finite) collection $\nuv = (\seq{\nu}1r)$ of order functions on $\OO_{X,0}$ with
values in the semigroups $\seq{S}1r$ respectively defines the $r$-index filtration
$$
J(\vv) = \{g\in \OO_{X,0} : \nu_i(g)\ge v_i \mbox{ for all } i=1,\ldots,r\}\,,
$$
where $\vv=(\seq v1r)\in \S = S_1\times \cdots \times S_r$. (The semigroup
$\S$ is partially ordered by
$\vv=(\seq v1r)\ge \ww=(\seq w1r)$ if and only if $v_i\ge w_i$ for all $i=1,\ldots,r$).

\begin{definition}
The {\em Hilbert function} $h$ of the filtration $\{J(\vv)\}$ is the function on $\S$ defined by
$$
h(\vv) = \dim \OO_{X,0}/J(\vv)\; .
$$
\end{definition}

We permit $h(\vv)$ to be equal to infinity.

The set $\Z[[\S]]$ of power series on the semigroup $\S$ is the set of formal
expressions of the form $\sum\limits_{\vv\in \S}a_{\vv} \tt^{\vv}$
($\vv=(\seq v1r)\in \S$, $a_{\vv}\in \Z$). One can write
$\tt^{\vv}$ as $t_1^{v_1}\cdots t_r^{v_r}$.
The set $\Z[[\S]]$ is a free abelian group. It is a ring if in each semigroup $S_i$,
each element $a\in S_i$ has only a finite number of different
representations as the sum
$a=a_1+a_2$ of two elements of $S_i$. This takes place, in particular, if $\nu_i$ is
a valuation on $\OO_{X,0}$ and $S_i$ is the semigroup of its values.

\begin{definition}
The {\em Hilbert series} $\widetilde{H}(\tt)$ of the filtration $\{J(\vv)\}$ is the generating series
of the values of the Hilbert function:
$$
\widetilde{H}(\tt) = \sum_{\vv\in \S} h(\vv)\;\tt^{\,\vv}\; .
$$
The Hilbert series is defined if all the values $h(\vv)$ are finite.
\end{definition}

For $I\subset I_0 = \{1,\ldots,r\}$ and $\vv\in \S$, let
$$
J^{+I}(\vv):=\{g\in J(\vv): \nu_i(g)>v_i \mbox{ for all } i\in I\}\,,
$$
$J^{+}(\vv):=J^{+I_0}(\vv)$.

\begin{definition}
The {\em Poincar\'e series} $P(\tt)$ of the filtration $J(\vv)$ is the element of
$\Z[[\S]]$ defined by
$$
P(\tt)=\sum_{\vv\in\S}\left(\sum_{I\subset  I_0}(-1)^{\#
I}\dim{J^{+I}(\vv)/J^{+}(\vv)}\right)\tt^{\,\vv}\,.
$$
\end{definition}

One can show that, for $S_i=\Z_{\ge 0}$, $i=1,\ldots,r$, this
definition coincides with (\ref{eq0}): see, e.g., \cite{IJM2003}.
The Poincar\'e series $P(\tt)$ is defined if all the factor spaces $J(\vv)/J^{+}(\vv)$ are finite
dimensional.

Let $Y(\vv):=\{g\in\OO_{X,0}:\nunu(g)=\vv\}$. One has
\begin{equation}\label{exact}
Y(\vv)=J(\vv)\setminus \bigcup_{i=1}^r J^{+\{i\}}(\vv)\,.
\end{equation}
Let $F_{\vv}= Y(\vv)/J^{+}(\vv)$. The union of the spaces
$F_{\vv}$ as a graded space, i.e. a space with the components
numbered by the elements of $\S$, is called {\em the extended
semigroup of the filtration} $\{J(\vv)\}$: see, e.g., \cite{Monats}.
The components $F_{\vv}$ of the extended semigroup are called
its {\em fibres}. The space $\coprod_{\vv\in\S}F_{\vv}$ is really
a semigroup if
all the order functions $\nu_i$ are valuations. However we keep
the name for the general case as well.

Let $S$ be the subset of $\S$ consisting of $\vv$ with $Y(\vv)\ne\emptyset$.
If $\nu_i$ are valuations, the set $S$ is a semigroup: the {\em semigroup of values}
of the collection $(\nu_1, \ldots, \nu_r)$.

One can see that the coefficient at $\tt^{\,\vv}$ in the
Poincar\'e series $P(t)$ is equal to the Euler characteristic
$\chi(\P F_{\vv})$ of the projectivization
$\P F_{\vv}=(F_{\vv}\setminus \{0\})/\C^{*}$ of the fibre of the extended
semigroup.

Let $K_0(\Var)$ be the Grothendieck ring of
quasi-projective varieties. It is generated by classes $[X]$ of
such varieties subject to the relations:\newline
1) if $X_1\cong X_2$, then $[X_1]=[X_2]$;\newline
2) if $Y$ is Zariski closed in $X$, then $[X]=[Y]+[X\setminus Y]$\newline
(the multiplication is defined by the Cartesian product).
Let $\LLL=[\C]$ be the class of the complex affine line in
$K_0(\Var )$. The natural map $\Z[\LLL]\to K_0(\Var)$
is an embedding.

Since the projectivization $\P F_{\vv}$ of the fibre $F_{\vv}$ of the extended semigroup
is the complement to an arrangement of projective subspaces in a projective space,
its class $[\P F_{\vv}]$ in the Grothendieck ring $K_0(\Var)$ is a polynomial in $\LLL$:
$[\P F_{\vv}]=p_{\vv}(\LLL)$.

Let $K_0(\Var)_{(\LLL)}$ be the localization of the Grothendieck ring $K_0(\Var)$
at $\LLL$ and let $q=\LLL^{-1}\in K_0(\Var)_{(\LLL)}$. The natural map $\Z[q,q^{-1}]\to K_0(\Var)_{(\LLL)}$
is an embedding.

\begin{definition}
The {\em generalized Poincar\'e series} of the filtration
$\{J(\vv)\}$ is the element of $\Z[q][[\S]]$ defined by
\begin{equation}\label{genPoincare}
P_g(\tt; q)
=\sum_{\vv\in\S}q^{h^{+}(\vv)-1}p_{\vv}(q^{-1})\;\tt^{\,\vv}\,,
\end{equation}
where $h^{+}(\vv)=dim \OO_{X,0}/J^{+}(\vv)$.
\end{definition}

In \cite{Monats}, there was described a relation of this
definition with a motivic measure on the projectivization
$\P\OO_{X,0}$ of the ring $\OO_{X,0}$ and an integral with
respect to it (for, so called, finitely determined order
functions). This identifies this definition with the one in (\ref{gen}).

Assume that, for each $\vv\in \S$, $\dim \OO_{X,0}/J(\vv)< \infty$ (and therefore
the Hilbert series and the generalized Poincar\'e series are defined).

\begin{proposition}
The Hilbert function, the Hilbert series and the generalized Poincar\'e series are equivalent
in the sense that each of them determines the other two.
\end{proposition}

\begin{proof}
The Hilbert function determines the generalized Poincar\'e series by the equation
(\ref{genPoincare}). For $\vv\in S$, the coefficient at $\tt^{\vv}$ in the
generalized Poincar\'e series
$P_g(\tt; q)$ is not zero and the lower degree in $q$ of it is
equal to
$h(\vv)$. Therefore
the generalized Poincar\'e series  $P_g(\tt; q)$ determines
$h(\vv)$ for all $\vv\in S$.
For $\vv\in \S\setminus S$, $h(\vv)$ is equal to the minimum of $h(\ww)$ for all
$\ww\ge \vv, \ww\in S$.
\end{proof}

In \cite{Monats}, there was considered one more generalized Poincar\'e series
corresponding to a filtration.

\begin{definition}
 The {\em generalized semigroup Poincar\'e series} of the
filtration $\{J(\vv)\}$ is the element of $\Z[\LLL][[\S]]$
defined by
\begin{equation}\label{gensemiPoincare}
 \widehat{P}_g(\tt; \L)= \sum_{\vv\in\S}[\P F_{\vv}]\tt^{\vv}
\left(=\sum_{\vv\in\S}p_{\vv}(\LLL)\tt^{\vv}\right)\,.
\end{equation}
\end{definition}

As in the generalized Poincar\'e series $P_g(\tt; q)$, the
coefficient at $\tt^{\,\vv}$
in the generalized semigroup Poincar\'e series
$\widehat{P}_g(\tt; \L)$ is different from zero
if and only if $\vv\in S$.

\section{Plane valuations and topology}\label{sec3}
The classification of plane valuations (i.e. valuations on the ring
$\OO_{\C^2,0}$) can be found in \cite{Spiv} or \cite{IJM2010}. We will follow the
assumptions and the terminology from \cite{IJM2010}.
In particular we assume that each valuation $\nu$ of rank one
is normalized by the requirement
that $\min\limits_{f\in \m}\nu(f)=1$, where $\m$ is the maximal ideal
of $\OO_{\C^2,0}$. For a curve valuation of rank two (type II.1
in \cite{IJM2010}) this
will be the normalization of the second component. For the
remainning valuations of rank two (exceptional curve
valuations: case II.2 of \cite{IJM2010})
we shall demand this for the first component of the value.

The notion of the minimal resolution of a valuation was described in \cite{IJM2010}. The
minimal resolution of a divisorial valuation is the minimal modification by a
(finite) sequence of blowing-ups which produces the corresponding divisor. A rank 1
plane valuation can be defined as the limit of a sequence of divisorial valuations:
see \cite{IJM2010}. The minimal resolution of this valuation is the projective limit of
the corresponding minimal resolutions

\begin{remark}
Strictly speaking this is not a resolution in the usual sense
(a modification of the plane) since the corresponding
map is not proper. For example the minimal resolution of a curve valuation of rank 1
is the resolution of the curve followed by the infinite sequence of blowing-ups at
the intersection points of the strict transform of the curve with the exceptional
divisor. The minimal resolution of a curve valuation of rank 2 is the minimal
resolution of the corresponding curve valuation of rank 1. The minimal resolution of
an exceptional curve valuation (of rank 2) is the minimal resolution of the
corresponding divisorial valuation (the first component of the considered one)
followed by the additional blowing-up at the corresponding intersection point.
\end{remark}

The minimal resolution of  a finite collection of valuations is
the projective
limit of the corresponding (multi-index) system of modifications. It is simply the
birational join of the minimal resolution of all the valuations.

The minimal resolution of a collection of plane valuations can be
described by its dual graph $G$. Vertices of $G$ correspond to the irreducible
components of the exceptional divisor. The set $\Gamma$ of vertices is a partially
ordered set: a component $E_{\sigma}$ of the exceptional divisor is greater than a
component $E_{\delta}$ if any modification containing $E_{\sigma}$ contains
$E_{\delta}$ as well. Two vertices of $G$ are connected by an edge if the
corresponding components of the exceptional divisor intersect. The graph $G$ is a tree.

\begin{remark}
There is another natural partial order on the set of vertices of the graph $G$
defined by the position of a vertex with respect to the root, i.e. to the vertex
corresponding to the first born component
of the exceptional divisor. A vertex $\sigma$ of the graph $G$ is greater than a
vertex $\delta$ if $\delta$ lies on the geodesic between the root and the vertex $\sigma$.
This partial order will be used in the proofs below, but not in the following definition.
\end{remark}

\begin{definition}
Two collections of plane valuations are called {\em topologically equivalent} if the dual
graphs of their minimal resolutions are isomorphic as graphs with partially ordered
sets of vertices and with marked vertices corresponding to divisorial and
exceptional curve valuations.
\end{definition}

\begin{remark}
For plane curve valuations this definition coincides with the
usual definition of the topological equivalence of plane curve
singularities.
According to this definition the curve valuations and the formal curve
valuations are topologically indistinguishable. They have also equal Hilbert functions
and generalized Poincar\'e series. In fact curve valuations and formal curve
valuations constitute one single class of valuations on the ring of formal power
series in two variables.
\end{remark}

For convenience we shall exclude curve rank 2 valuations since, strictly speaking,
the Hilbert function and the generalized Poincar\'e series are not defined for them
(see the Remark after the Proof of Theorem 1).

\begin{theorem}\label{theo1}
Let $(\seq{\nu}1r)$ be a collection of plane valuations does not containing curve
valuations of rank 2. Then the generalized Poincar\'e series $P_g(\tt,q)$ determines
the topological type of the collection.
\end{theorem}

Together with Proposition 1 this implies the following statement.

\begin{corollary}
The Hilbert function of a collection of plane valuations determines the topological
type of the collection.
\end{corollary}

\begin{Proof} {\bf of Theorem \ref{theo1}.}
As it was explained in Section~\ref{sec1}, the generalized Poincar\'e
series can be defined as the integral of the monomial function
$\tt^{\,\nunu}$ over the projectivization $\P\OO_{X,0}$ of the ring
$\OO_{X,0}$ with respect to the generalized Euler characteristic
(with values in the localization $K_0(\Var)_{(\LLL)}$ of the Grothendieck ring
$K_0(\Var)$ at $\LLL=q^{-1}$).

Let $\M$ be the completion of the localization
$K_0(\Var)_{(\LLL)}$ of the Grothendieck ring $K_0(\Var)$
with respect to the dimension filtration: see, e.g., \cite{DL}.
One has the natural map $K_0(\Var)_{(\LLL)} \to \M$.
Its restriction to $\Z[[q]]$ is an embedding.
We shall consider coefficients of the Poincar\'e series as elements of $\M$.
Under this assumption the
formula for the generalized Poincar\'e series  as an integral
gives the following statement. Let $I$ be a subset of the set
$I_0=\{1,\ldots,r\}$ of indices and let
$P_{g,\{\nu_i\}_{i\in I}}(\tt; q)$ be the generalized Poincar\'e
series of the subcollection $\{\nu_i\}_{i\in I}$ of the
collection $(\seq{\nu}1r)$. Then one has the following
``projection formula":
$$
P_{g,\{\nu_i\}_{i\in I}}(\tt; q) = P_{g}(\tt; q)|_{t_j=1 \mbox{ for
} j\notin I}\;.
$$
This implies that the generalized Poincar\'e series of a
collection of plane valuations determines the generalized
Poincar\'e series of any subcollections, in particular, the
generalized Poincar\'e series of each of them. On the other hand,
the generalized Poincar\'e series of a valuation determines the
usual Poincar\'e series of it:
$$
P(t) = P_g(t; 1)\;.
$$

In \cite{IJM2010} if was shown that the Poincar\'e series of a plane
valuation determines the dual graph of the minimal resolution of
it. Therefore, in order to describe the dual graph of the minimal
resolution of a collection of plane valuations, one has to
determine the splitting point $\alpha$ (i.e. the last
commond point counted
from the root of the tree) of the resolution trees
for each pair, say, $\nu_1$ and  $\nu_2$ of the valuations.
According of the description above the generalized Poincar\'e
series $P_g(\tt; q)$ determines the generalized Poincar\'e series
$P_{g}(t_1,t_2;q)$ of the pair $(\nu_1, \nu_2)$.

In turn the generalized Poincar\'e series $P_g(t_1,t_2; q)$
determines the semigroup of values of the pair $(\nu_1,\nu_2)$:
this is simply the set of values $(v_1,v_2)$ with non-zero
coefficients at the corresponding monomials. Now the splitting
vertex of the resolution graph $G$ corresponds to the minimal
element $(v^0_1,v^0_2)$ in the semigroup of values of the pair
$(\nu_1,\nu_2)$ such that $v_1^0=v_2^0$ (pay atention to the
normalization of the valuations described at the beginning of
the section) and there exists an element
$(v_1,v_2)$ of the
semigroup different from
$(v^0_1,v^0_2)$ and such that either $v_1=v^0_1$ or $v_2=v_2^0$.
The above fact for curve valuations of rank one can be found in
\cite{Del}, Remark (3.20). From the case of curve valuations
it is easy to extend it to divisorial valuations (considering the valuations defined by the corresponding curvettes,
i.e., the pull-downs of non-singular curves transversal to the divisors)
and, taking into account that any valuation could be regarded as a limit of divisorial
ones, one can deduce the result for the general case.

This proves the statement.
\end{Proof}

\begin{remark}
It is not difficult to see that one can permit curve rank 2
valuations in Theorem \ref{theo1} and in the Corollary if one
assumes that the Hilbert function $h(\vv)$ is defined with
infinite values at the elements $\vv$ with non-zero first
component of such a valuation and the coefficient at the
corresponding monomial $\tt^{\vv}$ in the generalized Poincar\'e
series is equal to zero. (The last assumption means that
$q^{\infty}$ should be considered as $0$.)
\end{remark}

\medskip

Though it is not clear whether the generalized
semigroup Poincar\'e series of a filtration determines its
Hilbert function (ans thus the generalized Poincar\'e series),
for plane valuations one can show that they are equivalent. This
follows from the following statement (together with the fact
that, for plane valuations, all three invariants: the Hilbert
function and the both generalized Poincar\'e series are
invariants of the topological type).

\begin{theorem}\label{theo2}
The generalized semigroup Poincar\'e series
$\widehat{P}_{g}(\tt;\L)$ determines the topological type of a
collection $(\seq{\nu}1r)$ of plane valuations.
\end{theorem}

\begin{proof}
The generalized semigroup Poincar\'e series determines the
semigroup of values of the collection $(\seq{\nu}1r)$ (as the
set of values $\vv$ with non-zero coefficients at $\tt^{\vv}$)
and therefore the semigroup of values of any subcollection of it.
In particular it determines the semigroup of values of each
valuation itself. In its turn this semigroup determines the
topological type of the valuation in all cases except the
following two:
\newline
1) The semigroup of values are the same for a plane curve
valuation and for a divisorial valuation whose resolution graph
is obtained from the resolution graph of the plane curve valuation
by a truncation of the last infinite tail.
\newline
2) Two divisorial valuations whose resolution graphs differs only
by the length of the last tail have the same semigroup of values.

If the collection consists only of one valuation ($r=1$), the
generalized semigroup Poincar\'e series determines the usual
Poincar\'e series which in turn determines the topological type.

If there are more than two valuations, then, say, the valuation
$\nu_1$ is a curve valuation if and only if there exists
$\seq{v^0}2r$ such that there are infinitely many elements in the
semigroup of values of the form
$(v_1,\seq{v^0}2r)$.
The second problem arises in a particular situation which will be
discuss later.

In order to restore the resolution graph one has to find the
splitting point for each pair of valuations. The method
to detect it from the semigroup of values of the pair of
valuations is described in the Proof of Theorem \ref{theo1}.

The only remaining problem (formulated as 2) above) is the
following one: assume that, say, $\nu_1$ is a divisorial
valuation and the corresponding divisor $E_{\alpha}$ is not on
the resolution tree of the collection $(\seq{\nu}2r)$.
(If it is on that tree, it corresponds to the splitting point
with the other valuations and was detected on the previous step.)
In this case the vertex corresponding to this divisor is on an
isolated branch of the resolution tree. Let $g_{\alpha}$ be a
function defining a curvette at the divisor $E_{\alpha}$ in
\cite{FAOM} it was explained how one can determine the dual graph
of the resolution of a divisorial valuation from its semigroup of
values plus the value  $\nu_1(g_{\alpha})$.

In this paragraph for each exceptional curve valuation (of rank
two) as $\nu_j(g)$ we use the first component of the value and
for each curve valuation of rank two we use the second component.
For each
$j=2,\ldots,r$, the semigroup of values of the pair
$(\nu_1,\nu_j)$ lies in the half-plane
$v_j\ge \frac{\nu_j(g_{\alpha})}{\nu_1(g_{\alpha})} v_1$
and the point
$(\nu_1(g_{\alpha}), \nu_2(g_{\alpha}), \ldots,
\nu_r(g_{\alpha}))$ is in the semigroup of values of the
collection. Moreover this point is the first point of the
semigroup of values on the ray
$$
v_j =  \frac{\nu_j(g_{\alpha})}{\nu_1(g_{\alpha})} v_1\; \
\mbox{ for } j=2,\ldots,r \;
$$
(see e.g. \cite{FAOM}).
This proves the statement.
\end{proof}

Adresses:

A. Campillo and F. Delgado:
Universidad de Valladolid, Instituto de Investigaci\'on en
Matem\'aticas (IMUVA), 47011 Valladolid, Spain.
\newline E-mail: campillo\symbol{'100}agt.uva.es, fdelgado\symbol{'100}agt.uva.es

S.M. Gusein-Zade:
Moscow State University, Faculty of Mathematics and Mechanics, Moscow, GSP-1, 119991, Russia.
\newline E-mail: sabir\symbol{'100}mccme.ru

\end{document}